\documentclass[12pt, reqno]{amsart}
\usepackage[dvips]{graphicx, color}
\usepackage{amsmath, amssymb}
\usepackage{array}
\usepackage{float}
\usepackage{wrapfig}
\usepackage{booktabs}
\usepackage{ascmac}
\usepackage{fancybox}
\usepackage[dvipdfmx,svgnames]{xcolor}
\usepackage{tikz}
\usepackage{bm}
\usepackage{pxpgfmark}
\usepackage{ascmac}
\usepackage[all]{xy}
\usetikzlibrary{patterns}
\usetikzlibrary{intersections, calc}
\usetikzlibrary{quotes,angles}
\usetikzlibrary {arrows.meta}
\usetikzlibrary {bending}
\makeatletter
\newcommand{\xequal}[2][]{\ext@arrow 0055{\equalfill@}{#1}{#2}}
\def\equalfill@{\arrowfill@\Relbar\Relbar\Relbar}
\makeatother
\newcommand{\maru}[1]{\raise0.2ex\hbox{\textcircled{\scriptsize{#1}}}}

 \def\NZQ{\mathbb}               % the font for N,Z,Q,R,C

 \def\ZZ{{\NZQ Z}}
 \def\RR{{\NZQ R}}

 \def\Pc{{\mathcal P}}

 \def\opn#1#2{\def#1{\operatorname{#2}}} % to make operators
 \opn\chara{char} \opn\length{\ell} \opn\pd{pd} \opn\rk{rk}
 \opn\projdim{proj\,dim} \opn\injdim{inj\,dim} \opn\rank{rank}
 \opn\depth{depth} \opn\grade{grade} \opn\height{height}
 \opn\embdim{emb\,dim} \opn\codim{codim}
 
 \opn\Tr{Tr} \opn\bigrank{big\,rank}
 \opn\superheight{superheight}\opn\lcm{lcm}
 \opn\trdeg{tr\,deg}%\emph{
 \opn\reg{reg} \opn\lreg{lreg} \opn\ini{in} \opn\lpd{lpd}
 \opn\size{size} \opn\sdepth{sdepth}
 \opn\link{link}\opn\fdepth{fdepth}\opn\lex{lex}
 \opn\tr{tr}
 \opn\type{type}
 \opn\gap{gap}
 \opn\arithdeg{arith-deg}
 \opn\revlex{revlex}
 \opn\cut{cut}
 \opn\div{div} \opn\Div{Div} \opn\cl{cl} \opn\Cl{Cl}
 \opn\Spec{Spec} \opn\Supp{Supp} \opn\supp{supp} \opn\Sing{Sing}
 \opn\Ass{Ass} \opn\Min{Min}\opn\Mon{Mon}
 \opn\Ann{Ann} \opn\Rad{Rad} \opn\Soc{Soc}
 \opn\Im{Im} \opn\Ker{Ker} \opn\Coker{Coker} \opn\Am{Am}
 \opn\Hom{Hom} \opn\Tor{Tor} \opn\Ext{Ext} \opn\End{End}
 \opn\Aut{Aut} \opn\id{id}
 
 \opn\nat{nat}
 \opn\pff{pf}%   \pf exists already
 \opn\Pf{Pf} \opn\GL{GL} \opn\SL{SL} \opn\mod{mod} \opn\ord{ord}
 \opn\Gin{Gin} \opn\Hilb{Hilb}\opn\sort{sort}
 \opn\PF{PF}\opn\Ap{Ap}
 \opn\mult{mult}
 \opn\bight{bight}
 \opn\aff{aff}
 \opn\relint{relint} \opn\st{st}
 \opn\lk{lk} \opn\cn{cn} \opn\core{core} \opn\vol{vol}  \opn\inp{inp} \opn\nilpot{nilpot}
 \opn\link{link} \opn\star{star}\opn\lex{lex}\opn\set{set}
 \opn\width{wd}
 \opn\Fr{F}
 \opn\QF{QF}
 \opn\G{G}
 \opn\type{type}\opn\res{res}
 \opn\conv{conv}
 \opn\Deg{Deg}
 \opn\Sym{Sym}
 \opn\gr{gr}
 
 \def\pot#1#2{#1[\kern-0.28ex[#2]\kern-0.28ex]}

 \opn\dirlim{\underrightarrow{\lim}}
 \opn\inivlim{\underleftarrow{\lim}}
 \def\Implies{\ifmmode\Longrightarrow \else
         \unskip${}\Longrightarrow{}$\ignorespaces\fi}
 \def\implies{\ifmmode\Rightarrow \else
         \unskip${}\Rightarrow{}$\ignorespaces\fi}
 \def\iff{\ifmmode\Longleftrightarrow \else
         \unskip${}\Longleftrightarrow{}$\ignorespaces\fi}

 \let\:=\colon
 \newtheorem{Theorem}{Theorem}[section]
 \newtheorem{Lemma}[Theorem]{Lemma}

 \newtheorem{Example}[Theorem]{Example}

 \let\epsilon\varepsilon
 \let\kappa=\varkappa
 \def\qed{\ifhmode\textqed\fi
       \ifmmode\ifinner\quad\qedsymbol\else\dispqed\fi\fi}
 \def\textqed{\unskip\nobreak\penalty50
        \hskip2em\hbox{}\nobreak\hfil\qedsymbol
        \parfillskip=0pt \finalhyphendemerits=0}
 \def\dispqed{\rlap{\qquad\qedsymbol}}
 
 \opn\dis{dis}
 \def\pnt{{\raise0.5mm\hbox{\large\bf.}}}
 
 \opn\Lex{Lex}

\begin{document}

\title{The minimal volume of a lattice polytope}
\author {Ichiro Sainose, Ginji Hamano, Tatsuo Emura and Takayuki Hibi}
\address{Ichiro Sainose, 
Hiroshima Municipal Motomachi Senior High School, 
Hiroshima, 730-0005, Japan}
\email{sainose23517257@nifty.com}
\address{Ginji Hamano, 
School of Science and Engineering, Tokyo Denki University, 
Saitama 350-0394, Japan}
\email{18hz002@ms.dendai.ac.jp}
\address{Tatsuo Emura,
Department of Pure and Applied Mathematics, Graduate School of Information Science and Technology, Osaka University, Osaka 565-0871, Japan}
\email{skmj50023@ares.eonet.ne.jp}
\address{Takayuki Hibi, 
Department of Pure and Applied Mathematics, Graduate School of Information Science and Technology, Osaka University, Osaka 565-0871, Japan}
\email{hibi@math.sci.osaka-u.ac.jp}

\dedicatory{ }
\keywords{lattice polytope, triangulation, Castelnuovo polytope}
\subjclass[2020]{Primary 52B20}

\begin{abstract}
Let $\Pc \subset \RR^d$ be a lattice polytope of dimension $d$.  Let $b$ denote the number of lattice points belonging to the boundary of $\Pc$ and $c$ that to the interior of $\Pc$.  It follows from a lower bound theorem of Ehrhart polynomials that, when $c > 0$, the volume of $\Pc$ is bigger than or equal to $(dc + (d-1)b - d^2 + 2)/d!$.  In the present paper, via triangulations, a short and elementary proof of the minimal volume formula  is given.  
\end{abstract}

\maketitle

\section{Introduction}
Let $\Pc \subset \RR^d$ be a {\em lattice polytope} of dimension $d$.  In other words, $\Pc$ is a convex polytope of dimension $d$ each of whose vertices belongs to $\ZZ^d$.  A {\em lattice point} of $\RR^d$ is a point belonging to $\ZZ^d$.  Let $b = b(\Pc)$ denote the number of lattice points belonging to the boundary $\partial \Pc$ of $\Pc$ and $c = c(\Pc)$ that to the interior of $\Pc$.  It follows from the lower bound theorem of Ehrhart polynomials \cite{hibi} that, when $c > 0$, 
\begin{eqnarray}
\label{formula}
{\rm vol}(\Pc) \geq (d \cdot c(\Pc) + (d-1) \cdot b(\Pc) - d^2 + 2)/d!,
\end{eqnarray}
where ${\rm vol}(\Pc)$ is the (Lebesgue) volume of $\Pc$.  However, the argument done in \cite{hibi} is rather complicated with deep techniques on polytopes.  In the present paper a short and elementary proof of the minimal volume formula (\ref{formula}) will be given.  Pick's formula guarantees that, when $d = 2$, the inequality (\ref{formula}) is an equality \cite{pick}.

A lattice polytope $\Pc \subset \RR^d$ of dimension $d$ is called {\em Castelnuovo} \cite{kawaguchi} if the equality holds in (\ref{formula}).  A few remarks on Castelnuovo polytopes will be also stated.

\section{Minimal volume formula}
In general, let $\Pc \subset \RR^d$ be a convex polytope of dimension $d$ and $V \subset \Pc$ a finite set to which each of the vertices of $\Pc$ belongs.  A {\em triangulation} of $\Pc$ on $V$ is a collection $\Gamma$ of $d$-simplices (simplices of dimension $d$) for which
\begin{itemize}
\item
each vertex of each $d$-simplex $F \in \Gamma$ belongs to $V$;
\item
each $x \in V$ is a vertex of a $d$-simplex $F \in \Gamma$;
\item
if $F \in \Gamma$ and $G \in \Gamma$, then $F \cap G$ is a face of $F$ and of $G$; 
\item
$\Pc = \bigcup_{F \in \Gamma} F$.
\end{itemize} 
The existence of a triangulation of $\Pc$ on $V$ is guaranteed by \cite[Lemma 1.1]{stanley}.  Thus in particular, if $\Pc$ is a lattice polytope, then a triangulation of $\Pc$ on $\Pc \cap \ZZ^d$ exists.
 
\begin{Lemma}
\label{Chopin}
Let $\Pc \subset \RR^d$ be a convex polytope of dimension $d$ and $V \subset \Pc$ a finite set to which each of the vertices of $\Pc$ belongs.  Let $b(\Pc) = |V \cap \partial \Pc|$, where $\partial \Pc$ is the boundary of $\Pc$, and $c(\Pc) = |V \cap (\Pc \setminus \partial \Pc)|$, where $\Pc \setminus \partial \Pc$ is the interior of $\Pc$.  Suppose that $c(\Pc) > 0$.  Then there exists a triangulation $\Gamma_\Pc$ of $\Pc$ on $V$ with 
\[
|\Gamma_\Pc| \geq d \cdot c(\Pc) + (d-1) \cdot b(\Pc) - d^2 + 2.  
\]
\end{Lemma}

\begin{proof}
We construct the required triangulation $\Gamma_\Pc$ by induction on $d$.  Let $d \geq 3$.  Let $\Gamma$ be a triangulation of $\Pc$ on $V$.  Let $\Delta$ denote the set of those $F \cap \partial \Pc$ with $F \in \Gamma$ for which $F \cap \partial \Pc$ is a $(d-1)$-simplex.  Fix $G_0 \in \Delta$.  Remove $G_0 \setminus \partial G_0$ from $\partial \Pc$, and one can assume that $\Pc' = \partial \Pc \setminus (G_0 \setminus \partial G_0)$ is a simplex in $\RR^{d-1}$ of dimension $d - 1$ via a one-point compactification.  Furthermore, the number of points in $V$ belonging to the boundary of $\Pc'$ is $b(\Pc') = d$ and that to the interior of $\Pc'$ is $c(\Pc') = b(\Pc) - d$.  Since $b(\Pc) > d$, it follows that $c(\Pc') > 0$.  The induction hypothesis yields a triangulation $\Delta'$ of $\Pc'$ on $\Pc' \cap V$ for which
\[
|\Delta'| \geq (d - 1) \cdot (b(\Pc) - d) + (d - 2) \cdot d - (d-1)^2 + 2.
\]
Let $\Gamma^{(0)} = \Delta' \cup \{G_0\}$.  Then $\partial \Pc = \bigcup_{G \in \Gamma^{(0)}} G$.  

Let $x_1, \ldots, x_c$ denote the points in $V$ belonging to the interior of $\Pc$.  Now, set 
\[
\Gamma^{(1)} = \{ {\rm conv}(G \cup \{x_1\}) : G \in \Gamma^{(0)} \},
\]
where ${\rm conv}(G \cup \{x_1\})$ is the convex hull of $G \cup \{x_1\}$ in $\RR^d$, and $\Gamma^{(1)}$ is a triangulation of $\Pc$ on $V^{(1)} = (\partial \Pc \cap V) \cup \{x_1\}$.  Since $|\Gamma^{(1)}| = |\Gamma^{(0)}| = |\Delta'| + 1$, it follows that
\begin{eqnarray*}
|\Gamma^{(1)}| & \geq & (d - 1) \cdot (b(\Pc) - d) + (d - 2) \cdot d - (d-1)^2 + 3 \\
& = & d + (d - 1) \cdot b(\Pc) - d^2 + 2.
\end{eqnarray*}

Let $c \geq 2$ and $x_2 \in F$ with $F \in \Gamma^{(1)}$.  Let $F_0$ be the smallest face of $F$ with $x_2 \in F_0$.  Then $x_2$ belongs to the interior of $F_0$.  Let $e = \dim F_0$ and $y_0, y_1, \ldots, y_e$ the vertices of $F_0$.  Thus $1 \leq e \leq d$.  Let $\{G_1, \ldots, G_q\}$ denote the set of those $G \in \Gamma^{(1)}$ for which $F_0$ is a face of $G$ and, for each $1 \leq i \leq q$, write $W_i$ for the set of vertices of $G_i$.  It follows that, for each $1 \leq i \leq q$ and for each $0 \leq j \leq e$, 
\[
G_i^{(j)} = {\rm conv}((W_i \setminus \{y_j\}) \cup \{x_2\})
\]
is a $d$-simplex.  Now, it then turns out that
\[
\Gamma^{(2)} = (\Gamma^{(1)} \setminus \{G_1, \ldots, G_q\}) \bigcup \, \left(\bigcup_{1 \leq i \leq q, \, 0 \leq j \leq e}\{G_i^{(j)}\}\right) 
\]
is a triangulation of $\Pc$ on $V^{(2)} = (\partial \Pc \cap V) \cup \{x_1, x_2\}$.  Since $F_0 \not\subset \partial \Pc$, one can regard
\[
\bigcup_{i=1}^{q} {\rm conv}(\{W_i \setminus \{y_0, \ldots, y_e\}\})
\]
as a boundary of a convex polytope of dimension $d - e$.  In particular $q \geq d - e + 1$.  Hence
\begin{eqnarray*}
|\Gamma^{(2)}| & \geq & d + (d - 1) \cdot b(\Pc) - d^2 + 2 + (d - e + 1)e \\
& \geq & 2 \cdot d + (d - 1) \cdot b(\Pc) - d^2 + 2.
\end{eqnarray*}
Continuing the procedure yields a triangulation $\Gamma^{(c)}$ of $\Pc$ on $$V^{(c)} = (\partial \Pc \cap V) \cup \{x_1, \ldots, x_c\}$$with
\[
|\Gamma^{(c)}| \geq d \cdot c(\Pc) + (d - 1) \cdot b(\Pc) - d^2 + 2,
\]
as desired.
\hspace{12cm}
\end{proof} 

\begin{Example}
{\em 
The picture drawn below demonstrates the procedure of constructing the triangulation $\Gamma_\Pc$ in the proof of Lemma \ref{Chopin}.  Let $\Pc = ABCDE$ denote the pyramid over the quadrangle $BCDE$.  Let $V = \{A,B,C,D,E,y_1,x_1,x_2\}$ where $y_1$ belongs to the boundary of $\Pc$ and where each of $x_1$ and $x_2$ belongs to the interior of $\Pc$.  Combining $y_1$ with each of $B,C,D,E$ yields the triangulation $\Gamma^{(0)}$ of the boundary $\partial \Pc$ of $\Pc$.  Combining $x_1 \in \Pc \setminus \partial \Pc$ with each of $A,B,C,D,E$ and $y_1$ yields the triangulation $\Gamma^{(1)}$ of $\Pc$ on $V^{(1)} = \{A,B,C,D,E,y_1,x_1\}$.  Let $x_2$ belong to the interior of the triangle $F_0$ with the vertices $x_1=y_0, y_1, D=y_2$.  Combining $x_2$ with each of $y_0,y_1,y_2$ yields the triangulation of $F_0$ on $\{x_2, y_0, y_1, y_2\}$.  Finally, combining $x_2$ with each of $C$ and $E$ yields the triangulation $\Gamma^{(2)}$ of $\Pc$ on $V$. 
%%%%%%%%%%%%%%%%%%%%%%%%%%%%%%%%%%%%%%%%%%%%%%%%%%
\begin{figure}[h]
\label{}
\centering
\includegraphics[scale=0.8]{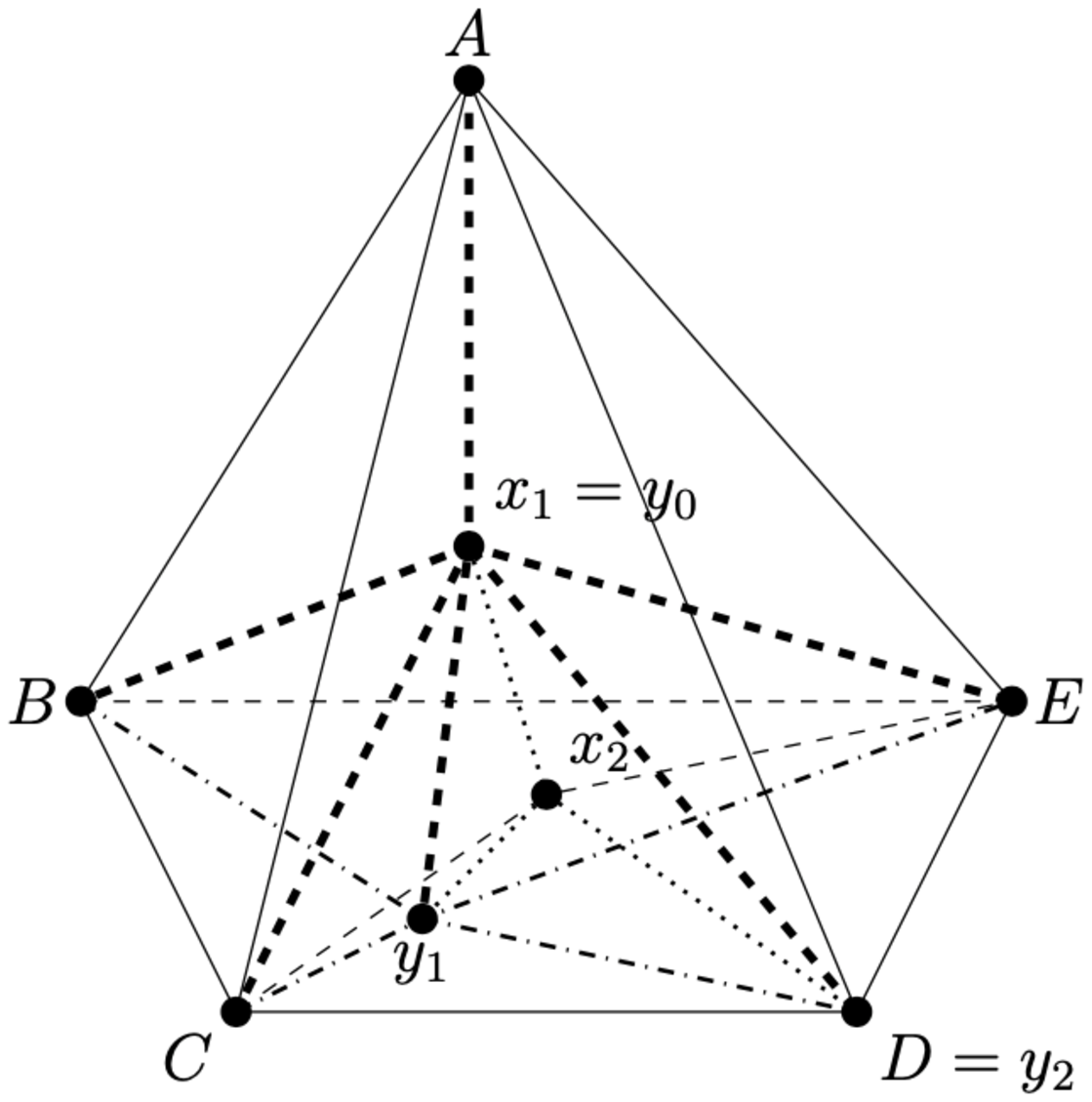}
\end{figure}
%%%%%%%%%%%%%%%%%%%%%%%%%%%%%%%%%%%%%%%%%%%%%%%%%%      
}
\end{Example}

We now come to the minimal volume formula (\ref{formula}).

\begin{Theorem}
\label{lower_bound}
Let $\Pc \subset \RR^d$ be a lattice polytope of dimension $d$.  Let $b(\Pc)$ denote the number of lattice points belonging to the boundary $\partial \Pc$ of $\Pc$ and $c(\Pc)$ that to the interior of $\Pc$.  Suppose that $c(\Pc) > 0$.  Then one has
\begin{eqnarray}
\label{formula*}
{\rm vol}(\Pc) \geq (d \cdot c(\Pc) + (d-1) \cdot b(\Pc) - d^2 + 2)/d!,
\end{eqnarray}
where ${\rm vol}(\Pc)$ is the (Lebesgue) volume of $\Pc$.
\end{Theorem}

\begin{proof}
Lemma \ref{Chopin} guarantees the existence of a triangulation $\Gamma_\Pc$ of $\Pc$ on $\Pc \cap \ZZ^d$ with
\begin{eqnarray}
\label{triangulation}
|\Gamma_\Pc| \geq d \cdot c(\Pc) + (d-1) \cdot b(\Pc) - d^2 + 2.
\end{eqnarray}
Since the volume of a lattice $d$-simplex of $\RR^d$ is a multiple of $1/d!$, the minimal volume formula (\ref{formula*}) follows from the inequality (\ref{triangulation}). \, \, \, \, \, \, \, \, \, \, \, \, \, \, \, \, \, \, \,  
\end{proof}

\section{Castelnuovo polytopes}
As before, let $\Pc \subset \RR^d$ be a lattice polytope of dimension $d$.  Following \cite{kawaguchi}, we say that $\Pc$ is {\em Castelnuovo} if $\Pc$ satisfies the equality of (\ref{formula}).  When $\Pc$ is Castelnuovo and when $V = \Pc \cap \ZZ^d$, the triangulation $\Gamma_\Pc$ constructed in the proof of Lemma \ref{Chopin} is unimodular.  (Recall that a triangulation $\Gamma_\Pc$ on $\Pc \cap \ZZ^d$ of a lattice polytope $\Pc \subset \RR^d$ of dimension $d$ is called {\em unimodular} if the volume of each of the $d$-simplices of $\RR^d$ belonging to $\Gamma_\Pc$ is $1/d!$.)  Furthermore, the triangulation $\Gamma_\Pc$ constructed in the proof of Lemma \ref{Chopin} is {\em regular}.  We refer the reader to \cite{HPPS} for fundamental materials on regular triangulations.  It then follows that every Castelnuovo polytope possesses a regular unimodular triangulation.   

It is reasonable to find all possible sequences $(d, b, c)$ of integers with $d \geq 3, \, b \geq d + 1, \, c \geq 1$ for which there exists a Castelnuovo polytope $\Pc \subset \RR^d$ of dimension $d$ with $b = b(\Pc)$ and $c = c(\Pc)$.

It follows from \cite{hibi_tsuchiya} that, given integers $d$ and $c$ with $d \geq 3$ and $c \geq 1$, there exists a Castelnuovo polytope (in fact, simplex) $\Pc \subset \RR^d$ of dimension $d$ with $b(\Pc) = d + 1$ and $c = c(\Pc)$.


\begin{thebibliography}{99}

\bibitem{HPPS}
C.~Haase, A.~Paffenholz, L.~C.~Piechnik and F.~Santos, Existence of unimodular triangulations -- positive results, {\em Memoirs of the American Mathematical Society} {\bf 270} (2021), \# 1321.

\bibitem{hibi}
T.~Hibi, A lower bound theorem for Ehrhart polynomials of convex polytopes, {\it Advances in Math.} {\bf 105} (1994), 162--165.

\bibitem{hibi_tsuchiya}
T.~Hibi and A.~Tsuchiya, Flat $\delta$-vectors and their Ehrhart polynomials, {\em Arch. Math. (Basel)} {\bf 108} (2017), 151--157.

\bibitem{kawaguchi}
R.~Kawaguchi, Sectional genus and the volume of a lattice polytope, {\em J. Algebraic Combin.} {\bf 53} (2021), 1253--1264.

\bibitem{stanley}
R.~Stanley, Decompositions of rational convex polytopes, {\em Annals of Discrete Math.} {\bf 6} (1980), 333--342.

\bibitem{pick}
D.~E.~Varberg, Pick's theorem revisited, {\em Amer. Math. Monthly} {\bf 92} (1985), 584--587.

\end{thebibliography}
\end{document}